\documentclass[10pt]{amsart}
\usepackage{amsthm}
\usepackage{amsmath}
\usepackage[dvipsnames]{xcolor}
\usepackage{tikz}
\usepackage{circledsteps}
\usepackage{enumitem}
\usepackage{amssymb}
\usepackage{todonotes}
\theoremstyle{plain}
\newtheorem{theorem}{Theorem}
\newtheorem*{mainthm}{Main Theorem}

\theoremstyle{definition}

\newtheorem{proposition}{Proposition}

\newtheorem{observation}[proposition]{Observation}
\newtheorem{comments}[proposition]{Discussion}
\newtheorem{corollary}[proposition]{Corollary}
\newtheorem{lemma}[proposition]{Lemma}
\newtheorem{definition}[proposition]{Definition}

\newtheorem{context}[proposition]{Context}
\theoremstyle{remark}
\newtheorem*{claim}{Claim}

\DeclareMathOperator{\ext}{ext}

\DeclareMathOperator{\otp}{otp}

\newcommand{\pone}{{\footnotesize{\sf Empty{ }}}}
\newcommand{\ptwo}{{\footnotesize{\sf Non-empty{ }}}}

\DeclareMathOperator{\ord}{On}

\DeclareMathOperator{\NS}{NS}
\newcommand{\sk}{\vskip.05in}

\newcommand{\pre}[2]{\vphantom{B}^{#2}#1}

\newcommand{\restr}{\!\upharpoonright\!}
\newcommand{\forces}{\Vdash}

\newcommand{\oneforces}[1]{\Vdash_{\!  {#1}}}

\DeclareMathOperator{\rk}{rk}

\newcommand{\lex}{\mathrm{lex}}

\numberwithin{equation}{section}
\numberwithin{proposition}{section}
\begin{document}
\title{Galvin's Conjecture and Weakly Precipitous Ideals}
\begin{abstract}
We investigate a combinatorial game on $\omega_1$ and show that mild large cardinal assumptions imply that every normal ideal on $\omega_1$ satisfies a weak version of precipitousness.  As an application,  we show that that the Raghavan-Todor\v{c}evi\'{c} proof~\cite{dilipstevo} of a longstanding conjecture of Galvin (done assuming the existence of a Woodin cardinal) can be pushed through under much weaker large cardinal assumptions.
\end{abstract}
\author{Todd Eisworth}
\email{eisworth@ohio.edu}
\thispagestyle{empty}
\thanks{This research was partly supported by NSF grant DMS-2400200. The author also thanks Tanmay Inamdar for pointing out an error in a previously circulated version of this manuscript.}
\maketitle

\section{Introduction}

\subsection{Motivation} This paper grows out of a recent result in Ramsey Theory due to Raghavan and Todor\v{c}evi\'{c}~\cite{dilipstevo}. From the existence of a Woodin cardinal, they establish a partition relation for uncountable sets of reals first conjectured by Galvin in the 1970s and prove
\begin{quote}
    If $X$ is an uncountable set of reals and we color the (un-ordered) pairs from $X$ with finitely many colors,  then $X$ contains a subset homeomorphic to the rationals on which at most two colors appear.
\end{quote}
We will refer to the above statement as {\em Galvin's Conjecture}, although in Baumgartner's~\cite{baum} the question is formulated only for the special case $X=\mathbb{R}$.

The value  ``two'' in the phrase ``at most two values'' is critical: a classical construction of Sierpi\'{n}ski~\cite{sierpinski} provides a coloring $c:[\mathbb{R}]^2\rightarrow \{0, 1\}$ that takes on both colors on any subset of $\mathbb{R}$ containing a subset order-isomorphic to the integers. Similarly, the restriction to uncountable sets of reals is necessary as well, as Baumgartner showed that there is a coloring of $[\mathbb{Q}]^2$ with countably many colors that takes on all values on any subset of $\mathbb{Q}$ homeomorphic to~$\mathbb{Q}$. In contrast, Galvin showed earlier in unpublished work that for any coloring of $[\mathbb{Q}]^2$ with finitely many colors, we can find a set of rationals that is order-isomorphic (though not necessarily homeomorphic) to $\mathbb{Q}$ on which the coloring assumes at most two values. Thus, Galvin's Conjecture holds if we replace ``homeomorphic to $\mathbb{Q}$'' with the weaker requirement ``order-isomorphic to $\mathbb{Q}$'' , and Baumgartner's work shows that there is potentially a big difference between the two statements.

For more general topological spaces, unpublished work of Todor\v{c}evi\'{c} and Weiss shows that the result of Baumgartner mentioned above can be extended to the broader class of $\sigma$-discrete metric spaces:  if $X$ is a $\sigma$-discrete metric space, then there is a function $c:[X]^2\rightarrow\omega$ that takes on every value on any subset of $X$ homeomorphic to $\mathbb{Q}$.  The authors of~\cite{dilipstevo} were also able to extend their positive answer to Galvin's Conjecture to a much more extensive (and essentially optimal) class of topological spaces as long as there are enough large cardinals around.
\begin{theorem}[Raghavan-Todor\v{c}evi\'{c}~\cite{dilipstevo}]
If there is a proper class of Woodin cardinals (or a single strongly compact cardinal) then the following statements are equivalent for a metrizable space $X$
\begin{itemize}
    \item $X$ is not $\sigma$-discrete.
    \sk
    \item Given $c:[X]^2\rightarrow l$ with $l<\omega$ there is a $Y\subseteq X$ homeomorphic to $\mathbb{Q}$ on which $c$ assumes at most two values.
\end{itemize}
\end{theorem}
Under the same large cardinal assumptions, they show that the second statement in the above theorem holds for the broader class of regular spaces with a point-countable base that are not left-separated.

\subsection{New Results}

It is natural to ask about the role of large cardinals in their result.   Their proof implicitly relies on the the existence of a generic elementary embedding with  well-founded target model, so it is natural to ask if one can make do with such an embedding whose target model is known only to have a sufficient long well-founded initial segment because such embeddings exist under much weaker assumptions.    We show this is possible by connecting their proof back to some old ideas of Shelah~\cite{111}.   What we do here is show that under weak large cardinals assumptions that normal ideals on $\omega_1$ automatically possess a weak form of precipitousness that is defined in terms of a two-player game, and then show how their construction can be modified to work with this game. The point is that our game is played with objects of small rank rather than with conditions in a large notion of forcing, and this allows us to reach the needed conclusion from weaker assumptions than their proof requires.

Note that Inamdar~\cite{inamdar} has recently announced that the Raghavan-Todor\v{c}evi\'{c} Theorem can be obtained in {\sf ZFC}, so this particular use of our techniques is potentially less interesting than when the work was originally done in the summer of~2023. In light of this development, we will limit our attention to Galvin's original question concerning uncountable sets of reals rather than obtaining the most general result. This allows us to demonstrate how our approach works while minimizing the background required of a potential reader, and to focus the exposition of the theorem in its most concrete setting. Thus, we present a complete proof of the following:

\begin{mainthm}
\label{mainresult}
If there is a Ramsey cardinal (or more generally, a generic elementary embedding with critical point $\omega_1$ that is well-founded out to the image of $\beth_2(\omega_1)^+)$)  then for any uncountable set of reals $X$ and any $c:[X]^2\rightarrow l$ with $l<\omega$, there is a set $Y\subseteq X$ homeomorphic to $\mathbb{Q}$ for which the range of $c\restr [Y]^2$ contains at most two colors.
\end{mainthm}

\section{Ideals, Filters, and Generic Elementary Embeddings}

We assume the reader has a basic familiarity with normal filters and ideals on~$\omega_1$. If $J$ is an ideal on a set $X$, then we let $J^*$ denote the dual filter, and let $J^+$ denote the collection of $J$-positive subsets of $X$. If $A$ is a $J$-positive subset of $X$, then we define the restriction of $J$ to $A$ as
\begin{equation}
    J\restr A :=\{B\subseteq X: A\cap B\in J\}.
\end{equation}
Note that $J\restr A$ is an ideal on $X$, and $A$ is in the filter dual to $J\restr A$.

We will use phrases such as ``for $J$-almost every $x\in X$'' to mean that the set of such $x$ is in $J^*$, and similarly ``for a $J$-positive set of $x\in X$'' means that the set of such $x$ is in $J^+$.  We also abbreviate these statements using quantifiers, so for example if $\phi$ is some formula then
\begin{equation*}
    (\forall^J x\in X)\phi(x) \Longleftrightarrow \{x\in X: \phi(x)\}\in J^*,
\end{equation*}
and
\begin{equation*}
    (\exists^J x\in X)\phi(x) \Longleftrightarrow\{x\in X:\phi(x)\}\notin J.
\end{equation*}

\subsection{Generic elementary embeddings}

There is a vast literature concerning generic elementary embeddings in set theory, with Foreman's article~\cite{mf} serving as a comprehensive resource for much of the material.  Our arguments require only a vestigial piece of this theory, and we sketch what we need below.

Suppose $\mathbb{P}$ is a notion of forcing, and let $G$ be a generic subset of $\mathbb{P}$.  In the extension $V[G]$, the subsets of $\omega_1^V$ from the ground model are a Boolean algebra
\begin{equation}
    \mathcal{P}(\omega_1)^V = \{ A\subseteq \omega_1^V: A\in V\}.
\end{equation}
Note that the ground model's $\omega_1$ may be countable in the extension, but the Boolean algebra $\mathcal{P}(\omega_1)^V$ still makes sense as a subalgebra of $\mathcal{P}(\omega_1^V)$.

Next, assume that $\dot D$ is a $\mathbb{P}$-name for an ultrafilter on $\mathcal{P}(\omega_1)^V$, and let $D$ be the interpretation of $\dot D$ in $V[G]$.  Given functions $f$ and $g$ with domain $\omega_1^V$ from the ground model, we can check whether they agree on a set in $D$, and since $D$ measures any subset of $\omega_1^V$ from the ground model this induces an equivalence relation on such functions.  We let $f/D$ denote the equivalence of class of $f$.

Similarly, given two functions $f$ and $g$ from $V$ that map $\omega_1^V$ to the ordinals, we can compare them modulo $D$ and declare
\begin{equation}
    f\leq_D g \Longleftrightarrow\{\alpha<\omega_1: f(\alpha)\leq g(\alpha)\}\in D.
\end{equation}
Since $D$ is an ultrafilter on $\mathcal{P}(\omega_1)^V$, the relation is a linear pre-order of such functions, and it induces a linear ordering of the corresponding equivalence classes.

What will be important for us is guaranteeing that this linear ordering obtained above from the name $\dot D$ has a sufficiently long well-ordered initial segment, and this is where large cardinal assumptions are useful.  The following notation of Shelah~\cite{256} captures what we need.

\begin{definition}
Suppose $\mathbb{P}$ is a notion of forcing, and let $\dot D$ be a $\mathbb{P}$-name for an ultrafilter on the Boolean algebra $\mathcal{P}(\omega_1)^V$.
\begin{enumerate}
    \item $t = (\mathbb{P}, \dot D)$ is {\em pre-nice} if $\mathbb{P}$ has a strongest condition $1_{\mathbb{P}}$, and  for each $p\in \mathbb{P}$,
    \begin{equation}
        J^t_p:=\{A\subseteq\omega_1: p\forces \check A\notin \dot D\}
    \end{equation}
    is a normal ideal on $\omega_1$.\footnote{The ideal $J^t_p$ is {\em the hopeless ideal conditioned on $p$} in terminology introduced by Foreman, as it consists of those sets that have been locked out of $\dot D$ by the condition $p$. }
    \item Given a cardinal $\kappa$, we say that $t = (\mathbb{P},\dot D)$ is {\em nice to $\kappa$} if $t$ is pre-nice and
    \begin{equation}
        \oneforces{\mathbb{P}} (\pre{\kappa}{\omega_1})^V\text{ is pre-wellordered by }\leq_{\dot D}.
    \end{equation}
    \sk
    \item Given a function $f:\omega_1\rightarrow\ord$ we let $\ext(f)$ (the {\em extension } of $f$) be a $\mathbb{P}$-name for the set
    \begin{equation}
        \{g/\dot D:  g\in (\pre{\ord}{\omega_1})^V\wedge g<_{\dot D} f\}
    \end{equation}
    in the generic extension.
\end{enumerate}
\end{definition}

For set-theorists, if we use $\dot D$ to generate an elementary embedding $j:V\rightarrow \pre{V}{\omega_1}/\dot D$ in $V[G]$ then $f/\dot D$ will be a (possibly ill-founded) ordinal in the target model, and $\ext(f)$ is the collection of predecessors of $f$ in the corresponding order.

\begin{proposition}
Suppose $\mathbb{P}$ is a notion of forcing and $\dot D$ is a $\mathbb{P}$-name for an ultrafilter on the Boolean algebra $\mathcal{P}(\omega_1)^V$.
\begin{enumerate}
    \item Given a $p\in \mathbb{P}$, a set $A$ is $J_p$-positive if and only if there is a $q\leq p$ such that $q\forces \check{A}\in\dot D$.
    \sk
    \item $(\mathbb{P}, \dot D)$ is pre-nice if and only $\oneforces{\mathbb{P}}\text{``$\dot D$ is a $V$-normal ultrafilter on $\mathcal{P}(\omega_1)^V$''}$, that is, any regressive $f:\omega_1\rightarrow\omega_1$ in the ground model is forced (by every condition) to be constant on a set in $\dot D$.
    \sk
    \item $(\mathbb{P},\dot D)$ is nice to $\kappa$ if and only if for any $f:\omega_1\rightarrow\kappa$ in the ground model,
    \begin{equation*}
    \oneforces{\mathbb{P}}\text{``$\ext(f)$ is well-ordered by $<_{\dot D}$''}.
    \end{equation*}
    \item If $(\mathbb{P},\dot D)$ is nice to $\kappa$ and $f:\omega_1\rightarrow\kappa$ in the ground model, then the the set of conditions that decide the order-type of $\ext(f)$ is dense in $\mathbb{P}$.
\end{enumerate}
\end{proposition}
\begin{proof}
All of these are easy using properties of the forcing relation.
\end{proof}

\begin{comments}
For those with a more solid grounding in this material,  if $(\mathbb{P},\dot D)$ is nice to $\kappa$ then forcing with $\mathbb{P}$ adds a generic elementary embedding with critical point $\omega_1$ and a target model that is well-founded up to the image of $\kappa$.
If we identify the well-founded portion of the target model with its transitive collapse, then any function mapping $\omega_1$ to $\kappa$ in the ground model will map via the embedding to an actual ordinal, and the set of conditions that decide which ordinal this is will be dense in~$\mathbb{P}$.
\end{comments}

\section{A game from an embedding}

The Raghavan-Todor\v{c}evi\'c solution to Galvin's conjecture relies on playing a game with conditions in the countable stationary tower forcing $\mathbb{Q}_{<\delta}$ with $\delta$ a Woodin cardinal. The property of the game they use relies on the fact that this forcing adds a generic elementary embedding with critical point $\omega_1$ and a well-founded target model. In this section, we develop a game that is strong enough to allow their argument to be pushed through, but weak enough so that the needed properties hold in the presence of a Ramsey cardinal rather than requiring a Woodin.

\subsection{The game}

The game we want to use is a variation on an idea that has appeared several places in the literature. What we do here is to implement the {\em weakly precipitous } game of Jech~\cite{jech}, except we make use of a parameter.  The parameters we use are just families $\mathbb{J}$ of normal ideals on $\omega_1$ satisfying the following two conditions:
\begin{itemize}
    \item $\mathbb{J}$ has a $\subseteq$-minimal element $\min(\mathbb{J})$, and
    \sk
    \item for each $J\in\mathbb{J}$ and $J$-positive set $A$, there is an ideal in $\mathbb{J}$ extending $J\restr A$.
\end{itemize}
We will (temporarily) refer to such a collection $\mathbb{J}$ as a ``suitable parameter'' or just ``parameter'', and for each such parameter
$\mathbb{J}$ we will have an associated game $\Game(\mathbb{J})$. These parameters form the 2nd level of a hierarchy of objects used by Shelah in his study of generalizations of the Galvin-Hajnal theorem in cardinal arithmetic, and the collection of suitable parameters is denoted simply $OB_2$ in~\cite{256} and~\cite{cardarith}.  Donder and Levinksi use similar objects --- $\kappa$-plain collections of filters --- in their work~\cite{donder}.

Clearly the collection of ALL normal ideals on $\omega_1$ is a suitable parameter, and if $t = (\mathbb{P},\dot D)$ is pre-nice then the collection of ideals $\{J^t_p:p\in\mathbb{P}\}$ is a suitable parameter as well. Finally, note that if $\mathbb{J}$ is a suitable parameter and $J\in\mathbb{J}$, then $\{I\in\mathbb{J}: J\subseteq I\}$ is also a suitable parameter.

\begin{definition}
Given a suitable parameter $\mathbb{J}$, the game $\Game(\mathbb{J})$ is a contest of length~$\omega$ between two players \pone and \ptwo\!\!. At a stage $n$, \pone will be selecting a subset $A_n$ of $\omega_1$ and \ptwo will respond by choosing an ideal $J_n$ from $\mathbb{J}$.  The rules are as follows:
\begin{enumerate}
    \item For bookkeeping purposes, set $J_{-1}$ to be $\min(\mathbb{J})$.
    \sk
    \item Given $J_n$, \pone will select a $J_n$-positive set $A_{n+1}$, and \ptwo will respond by choosing some $J_{n+1}\in\mathbb{J}$ that extends $J_n\restr A_{n+1}$.
    \sk
    \item After $\omega$ stages, \pone is declared the winner if $\bigcap_{n<\omega}A_n=\emptyset$.
\end{enumerate}
\end{definition}

\begin{definition}
We say a that a family $\mathbb{J}$ of normal ideals on $\omega_1$ is {\em weakly precipitous} if $\mathbb{J}$ is a suitable parameter and \pone does not have a winning strategy in the game $\Game(\min(\mathbb{J}))$.  We say that $\mathbb{J}$ is {\em everywhere weakly precipitous} if $\mathbb{J}$ if $\{I\in\mathbb{J}:I\supseteq J\}$ is weakly precipitous for each $J\in\mathbb{J}$, and we refer to the corresponding version of the game simply as ``{\em playing $\Game(\mathbb{J})$ beyond $J$}'' rather than introducing more notation.
\end{definition}

\begin{comments}
If $J$ is a normal ideal on $\omega_1$ and we take $\mathbb{J}$ to be the parameter consisting of all normal ideals on $\omega_1$ that extend $J$, then we arrive at the {\em weakly precipitous} game of Jech~\cite{jech}, and he calls a normal ideal on $\omega_1$  {\em weakly precipitous} if \pone does not have a winning strategy in the corresponding game and uses an argument of Shelah to show that the non-stationary ideal on $\omega_1$ is weakly precipitous in the presence of a Ramsey cardinal.   Note that if $\mathbb{J}$ is a weakly precipitous family of ideals, then the ideal $\min(\mathbb{J})$ will be weakly precipitous in this sense.

The game is also related to the  $\mathbb{Q}_{<\delta}$-game used by Raghavan and Todor\v{c}evi\'c, and the {\em precipitous game} of Galvin characterizing precipitous ideals.  The difference is that these games involve making moves with conditions in a notion of forcing, while $\Game(\mathbb{J})$ works with the ``projections'' of forcing conditions on objects of small rank.  This is why we can get away with partially well-founded embeddings in our argument: the game is being played with small objects, but \ptwo can use the (potentially quite large) notion of forcing to defeat any purported winning strategy for \pone\!.

This connection to precipitous ideals shows us that it is possible for the ideal $\min(\mathbb{J})$ to be weakly precipitous while the family $\mathbb{J}$ itself is not.  For example, if $J$ is a normal ideal on $\omega_1$ and we define $\mathbb{J}$ to be the collection of ideals of the form $J\restr A$ with $A\in J^+$, then $\mathbb{J}$ is a suitable parameter with $\min(\mathbb{J})=J$, but the collection $\mathbb{J}$ is weakly precipitous if and only if the ideal $J$ is precipitous: for this choice of $\mathbb{J}$, the weakly precipitous game $\Game(\mathbb{J})$ reduces to the game used by Galvin in his characterization of precipitous ideals. As we shall shortly see, weakly precipitous ideals and families of normal ideals on $\omega_1$ exist in the presence of Ramsey cardinals, while it is well-known that precipitous ideals are equiconsistent with the existence of measurable cardinals.
\end{comments}

\begin{theorem}
\label{gamethm}
If $t=(\mathbb{P},\dot D)$ is nice to $\kappa = \beth_2(\omega_1)^+=(2^{2^{\aleph_1}})^+$, then $\mathbb{J}^t=\{J^t_p:p\in\mathbb{P}\}$ is an everywhere weakly precipitous family of ideals.
\end{theorem}

\begin{proof}
Given $J\in\mathbb{J}=\mathbb{J}^t$, assume by way of contradiction that \pone has a winning strategy in the version of $\Game(\mathbb{J})$ played past $J$.  For each $\alpha<\omega_1$, let $T_\alpha$ consist of all finite sequences $\sigma$ of odd length that comprise finite partial plays $\langle A_0, F_0, \dots, A_{n-1}, F_{n-1}, A_n\rangle$ in the game where \pone is using her winning strategy, but for which $\alpha$ has not yet been eliminated, that is, with $\alpha\in A_n$.  Since \pone is using a winning strategy, we know $T_\alpha$ is a well-founded tree, and thus there is a corresponding ranking function $\rho_\alpha$ of the elements of $T_\alpha$, defined in standard fashion via the recursion
\begin{equation}
    \rho_\alpha(\sigma) = \sup\{\rho_\alpha(\eta)+1: \text{$\eta$ an immediate successor of $\sigma$ in $T_\alpha$}\}.
\end{equation}
Note that there are at most $\beth_2(\omega_1)$ moves available to each player in total. This implies $|T_\alpha|=\beth_2(\omega_1)$, and since the range of our rank function is an initial segment of the ordinals,  we conclude
\begin{equation}
    \rk_\alpha(\sigma)<\beth_2(\omega_1)^+=\kappa
\end{equation}
for any $\alpha\in A_n$ and $\sigma\in T_\alpha$. This is a critical observation for our argument and explains the appearance of $\beth_2(\omega_1)^+$ in the statement of Theorem~\ref{gamethm}.

Fix $p\in\mathbb{P}$ with $J=J_p$.  The second player will work to defeat the winning strategy by building a decreasing sequence $\langle p_n:n<\omega\rangle$ of conditions in $\mathbb{P}$ below~$p$, arranging that his move $J_n$ will be the ideal $J_{p_n}\in\mathbb{J}$.  Assume we have been playing a game in which \pone is using her strategy, and we have arrived at the sequence $\sigma=\langle A_0, F_0, \dots, A_n, F_n, A_{n+1}\rangle$.   Since \pone is using her winning strategy, for each $\alpha\in A_{n+1}$ the sequence $\sigma$ is in the tree $T_\alpha$ and therefore has been assigned corresponding rank less than $\kappa$.
Putting this together gives us a single function function $\rho_{n+1}: \omega_1\rightarrow \kappa$ defined by
\begin{equation}
\rho_{n+1}(\alpha)=
\begin{cases}
    \rk_\alpha(\sigma)   &\text{if $\alpha\in A_{n+1}$, and}\\
    0   &\text{otherwise.}
\end{cases}
\end{equation}
 \ptwo now extends $p_n$ to a condition $q\in\mathbb{P}$ such that $J_q$ extends $J_n\restr A_{n+1}$,  and then chooses $p_{n+1}\leq q$ deciding the order-type of $\ext(\rho_{n+1})$.   Finally, he sets $J_{n+1} = J_{p_{n+1}}$ and the game continues.

We show this leads to a contradiction. Assume that \ptwo plays as described above, so at stage $n$ there is an ordinal $\delta_n$ such that
\begin{equation}
    p_n\forces \otp(\ext(\rho_n))=\delta_n.
\end{equation}
For each $\alpha\in A_{n+1}$ we know $\rho_{n+1}(\alpha)<\rho_n(\alpha)$ because by advancing our game we move further down the tree $T_\alpha$.  Since
\begin{equation}
    p_{n+1}\forces A_{n+1}\in \dot D,
\end{equation}
it follows that $p_{n+1}$ forces $\rho_{n+1}<_{\dot D}\rho_n$, and therefore $\delta_{n+1}<\delta_n$.  This leads to an immediate contradiction.
\end{proof}

We finish this section by noting that in the presence of a Ramsey cardinal, every normal ideal on $\omega_1$ is weakly precipitous. This result is implicit in Chapter IV of~\cite{cardarith}, although it is somewhat hidden by the generality in which Shelah works.  We assume more background knowledge than needed in the rest of the paper, but the proof consists of chaining together several standard facts from~\cite{mf} and~\cite{paul} about generalized stationary sets.   All  we will need to know about Ramsey cardinals is that they are {\em completely J\'onnson}:
 if $\lambda$ is a Ramsey cardinal and $S$ is a stationary subset of $[H(\theta)]^\omega$ for some sufficiently large regular $\theta<\lambda$, then the set
\begin{equation}
    T:=\{X\subseteq V_\lambda: T\cap H(\theta)\in S\text{ and }|T\cap \lambda|=\lambda\}
\end{equation}
is stationary (in the generalized sense) in $[V_\lambda]^\lambda$. Larson~\cite{paul} contains details about the relevance of completely J\'onsson cardinals to stationary tower forcing, and for a proof (Remark 2.3.3) that Ramsey cardinals satisfy this property.

\begin{theorem}
\label{needramsey}
If there is a Ramsey cardinal, then the set $\mathbb{J}$ of all normal ideals on~$\omega_1$ is everywhere weakly precipitous. In particular, if there is a Ramsey cardinal then every normal ideal on $\omega_1$ is weakly precipitous.
\end{theorem}
\begin{proof}
Suppose $J$ is a normal ideal on $\omega_1$, and let $\lambda$ be a Ramsey cardinal. It suffices to show that there is a pair $t=(\mathbb{P},\dot D)$ that is nice to $\lambda$ (which is certainly much larger than $\beth_2(\omega_1)^+$) such that $J = \min(\mathbb{J}^t)$. Working in $V_\lambda$, we apply a theorem due to Burke~\cite{burke} to represent $J$ as the canonical projection of the non-stationary ideal.  Following the presentation of Burke's theorem from Foreman~\cite{mf}, if $\theta$ is a sufficiently large regular cardinal, and we define $S_J$ to be the collection of countably elementary submodels of $H(\theta)$ containing $J$ for which
\begin{equation}
    M\cap\omega_1\notin\bigcup(M\cap J),
\end{equation}
then $S_J$ is a stationary subset of $[H(\theta)]^\omega$ and the canonical projection sending a model $M$ to the ordinal $M\cap\omega_1$ witnesses that $J$ is the Rudin-Keisler projection of the non-stationary ideal restricted to $S_J$.

Now we use the fact that $\lambda$ is completely J\'onsson to ``inflate'' the elements of $S_J$ to models of size $\lambda$: the set
\begin{equation}
    T_J:=\{X\subseteq V_\lambda:  X\cap H(\theta)\in S_J\text{ and }|X\cap\lambda|=\lambda\}
\end{equation}
is stationary in $[V_\lambda]^\lambda$.  Now we forces with let $\mathbb{P} = \mathcal{P}(T_J)/\NS$, and let $G$ be a generic subset of~$\mathbb{P}$. Standard arguments tell us that $G$ will project to a $V$-ultrafilter $D$ on~$\omega_1$, via the correspondence
\begin{equation}
    A \in G\Longleftrightarrow \{X\cap\omega_1: X\in A\}\in D.
\end{equation}
If we let $\dot D$ be a name for this object, then $t=(\mathbb{P},\dot D)$ is a pre-nice pair with $J$ equal to $\min(\mathbb{J}^t)$.

The argument that this pre-nice pair is nice to $\lambda$ is a venerable one, attributed to Magidor in~\cite{jech} and~\cite{111}. The fact that $|X\cap\lambda|=\lambda$ and $X\cap\omega_1<\omega_1$ for $X\in T_J$ is critical.  Given a function $f:\omega_1\rightarrow\lambda$, define a function $\hat{f}$ on $T_J$ by setting
\begin{equation}
    \hat{f}(X) = \text{the $f(X\cap\omega_1)^{\mathrm{th}}$ element of the increasing enumeration of $X\cap\lambda$}.
\end{equation}
Note that $\hat{f}(X)$ is an element of $X$ as $|X\cap\lambda|=\lambda$ and $X\cap\omega_1<\omega_1$. A density argument using the normality of the non-stationary ideal on $T_J$ tells us that $\hat{f}$ will be constant (say with value $\delta_f<\lambda$) on a set in $G$ for any $f:\omega_1\rightarrow\lambda$ in the ground model. Given two such functions $f$ and $g$ from $V$, if $f<_D g$ in $V[G]$, then $\hat{f}<_G\hat{g}$ and so $\delta_f<\delta_g$. From this we conclude that $t$ is nice to $\lambda$.
\end{proof}

\section{Saturated Pairs and Winning Ideals}

Our next move is to bring in a coloring and see how some of the arguments from~\cite{dilipstevo} can be extended to work with the framework from the last section.  We remind the reader of our promise to focus only on the case where we partition an uncountable set of reals. Thus, we fix the following context for our discussion.

\begin{context}
Fix a coloring $c:[\omega_1]^2\rightarrow l$ with $l<\omega$, and assume $\mathcal{T}$ is a topology on $\omega_1$ such that $X=(\omega_1,\mathcal{T})$ homeomorphic to a subset of $\mathbb{R}$.
\end{context}

Our initial steps are straightforward modifications of some definitions from~\cite{dilipstevo}.

\begin{definition}
Suppose $J$ is an ideal on $\omega_1$.
\begin{enumerate}
    \item Given  $\alpha<\omega_1$, a $J$-positive subset $B$ of $\omega_1$, and a color $i<l$, we say that $\alpha$ is $i$-large in $B$ with respect to $J$ if the set of $\beta\in B$ with $c(\alpha,\beta)=i$ is $J$-positive.
    \sk
    \item  A pair $\langle A, B\rangle$ of $J$-positive sets is said to be {\em weakly $\langle i, j\rangle$-saturated over $J$} if the set of $\alpha\in A$ that are $i$-large in $B$ with respect to $J$ is $J$-positive, and the set of $\beta\in B$ that are $j$-large in $A$ with respect to $J$ is $J$-positive, that is, if
\begin{equation}
\label{left}
    (\exists^J\alpha\in A)(\exists^J \beta\in B)\left[c(\alpha,\beta)=i\right],
\end{equation}
and
\begin{equation}
\label{right}
    (\exists^J \beta\in B)(\exists^J \alpha\in A)\left[c(\alpha,\beta)=j\right].
\end{equation}
\sk
\item A pair $\langle A, B\rangle$ of $J$-positive sets is {\em$\langle i, j\rangle$-saturated over $J$} if for any $J$-positive $C\subseteq A$ and $D\subseteq B$, the pair $\langle C, D\rangle$ is weakly $\langle i, j\rangle$-saturated over $J$.  We may say that $\langle C, D\rangle$ is a $J$-positive refinement of $\langle A, B\rangle$.
\end{enumerate}
\end{definition}

\begin{lemma}[Facts about weak saturation]
\label{basic}
Let $J$ be an ideal on $\omega_1$, and assume $A$ and $B$ are $J$-positive sets.
\begin{enumerate}
    \item There is a pair of colors $\langle i, j\rangle$ such that $\langle A, B\rangle$ is weakly $\langle i, j\rangle$-saturated over $J$.
    \sk
    \item If $\langle A, B\rangle$ is NOT weakly $\langle i, j\rangle$-saturated over $J$, then the same is true for any pair of $J$-positive sets $\langle C, D\rangle$ with $C\subseteq A$ and $D\subseteq B$.
\end{enumerate}
\end{lemma}
\begin{proof}
Left to reader.
\end{proof}

\begin{lemma}[Facts about saturation]
\label{satlemma}
Under the same assumptions as the previous proposition:
\begin{enumerate}
        \item If $\langle A, B\rangle$ is $\langle i, j\rangle$-saturated over $J$, then so is any $J$-positive refinement of $\langle A, B\rangle$.
        \sk
        \item If $\langle A, B\rangle$ is $\langle i, j\rangle$-saturated over $J$, then $J$-almost every $\alpha\in A$ is $i$-large in~$B$, and $J$-almost every $\beta\in B$ is $j$-large in $A$.
        \sk
        \item For some choice of $\langle i, j\rangle$ there is a pair of $\langle i, j\rangle$-saturated sets over $J$.
\end{enumerate}
\end{lemma}
\begin{proof}
The first statement is immediate from the definition, and easily implies the second.  If (3) fails, then for any pair $\langle A, B\rangle$ of $J$-positive sets and any pair of colors $\langle i, j\rangle$, we can find a $J$-positive refinement of $\langle A, B\rangle$ that is not weakly $\langle i, j\rangle$-saturated over $J$. Since this state of affairs is inherited by any further $J$-positive refinements, by making repeated extensions and working through all pairs of colors, we arrive at a pair of $J$-positive sets that fail to be weakly $\langle i, j\rangle$-saturated for EVERY choice of $\langle i, j\rangle$. But then this contradicts part (1) of the previous lemma, and we are done.
\end{proof}

The preceding lemmas are both taken almost directly from~\cite{dilipstevo}.  We need to extend things a bit by lifting the concept of $\langle i, j\rangle$-saturation from pairs of sets to pairs of normal ideals, as in the following definition:

\begin{definition}
Let $\langle i, j\rangle$ be a pair of colors. We say that a pair $\langle J_0, J_1\rangle$ of normal ideals on $\omega_1$ is $\langle i, j\rangle$-saturated if for any normal extensions $I_0$ and $I_1$  of $J_0$ and $J_1$ respectively, we have
    \begin{align}
        (\forall^{I_0}\alpha<\omega_1)(\exists^{I_1}\beta<\omega_1)&\left[c(\alpha,\beta)=i\right]\label{eqn?}\\
        \intertext{and}
        (\forall^{I_1}\beta<\omega_1)(\exists^{I_0}\alpha<\omega_1)&\left[c(\alpha,\beta)=j\right].
    \end{align}
\end{definition}

Now the key point: we can use normality to see that $\langle i, j\rangle$-saturated pairs of sets give rise to $\langle i, j\rangle$-saturated pairs of ideals in the natural way.

\begin{lemma}
\label{idealsat}
Suppose $J$ is a normal ideal on $\omega_1$ and $\langle A, B\rangle$ is a pair of $J$-positive sets that is $\langle i, j\rangle$-saturated over $J$.  Then $\langle J\restr A, J\restr B\rangle$ is an $\langle i, j\rangle$-saturated pair of normal ideals.
\end{lemma}
\begin{proof}
Suppose this fails as witnessed by ideals $J_A$ and $J_B$.  Assuming (without loss of generality) that the problem lies with (\ref{eqn?}), we can find a $J_A$-positive set $C$ such that for each $\alpha\in C$,
\begin{equation}
    B_\alpha:=\{\beta\in B: c(\alpha,\beta)\neq i\}\in J^*_B.
\end{equation}
Since $J_B$ is normal, we know the diagonal intersection $D$ of these sets $B_\alpha$ is in~$J_B^*$,~so
\begin{equation}
\label{weapon}
    \beta\in D\text{ and }\alpha\in C\cap \beta\Longrightarrow c(\alpha,\beta)\neq i.
\end{equation}
Now both $\langle C\cap A, D\cap B\rangle$ is a $J$-positive refinement of $\langle A, B\rangle$, and so forms an $\langle i, j\rangle$-saturated pair over~$J$. In particular, we can find $\alpha\in C$ such that $C_\alpha:=\{\beta\in D: c(\alpha,\beta)=i\}$ is $J$-positive.  But this implies that there are arbitrarily large $\beta\in D$ for which $c(\alpha,\beta)=i$,  which contradicts~(\ref{weapon}).
\end{proof}

\subsection{Using a generic elementary embedding}

Our next goal is to explore how our coloring interacts with a pre-nice pair~$t=(\mathbb{P},\dot D)$, and the corresponding family of ideals $\mathbb{J}^t=\{J^t_p:p\in\mathbb{P}\}$.

\begin{lemma}
Suppose $(\mathbb{P},\dot D)$ is pre-nice.  Then there is a $p\in\mathbb{P}$ and a pair of colors $\langle i, j\rangle$ such that the set of $q\in\mathbb{P}$ for which there is an $\langle i, j\rangle$-saturated pair of $J$-positive sets over $J_q$ is dense below $p$.
\end{lemma}
\begin{proof}
Suppose not.  Then for any condition $p\in\mathbb{P}$ and pair of colors $\langle i, j\rangle$ we can find a $q\leq p$  so that given any $r\leq q$ there is no $\langle i, j\rangle$-saturated pair over $J_r$.  Using this and by making successive extensions running through all the finitely many possible pairs $\langle i, j\rangle$, we arrive at a condition $q$ such that $J_q$ fails to have an $\langle i, j\rangle$-saturated pair for EVERY choice of $\langle i, j\rangle$, but this contradicts the last part of Lemma~\ref{satlemma}.
\end{proof}

\begin{corollary}
If $t=(\mathbb{P},\dot D)$ is nice to the cardinal $\kappa$ then we can find another such pair $t'=(\mathbb{P}',\dot D')$ and a pair of colors $\langle i, j\rangle$ such that for each $q\in \mathbb{P}'$ there is an $\langle i, j\rangle$-saturated pair of normal ideals over $J^{t'}_q$.
\end{corollary}
\begin{proof}
Fix a condition $p\in\mathbb{P}$ and colors $\langle i, j\rangle$ as in the previous lemma, and define
\begin{equation}
    \mathbb{P}'=\{q\in\mathbb{P}: q\leq p\text{ and there is an $\langle i, j\rangle$-saturated pair over }J_{q}\}.
\end{equation}
Since $\mathbb{P}'$ is dense below $p$ in $\mathbb{P}$, forcing with $\mathbb{P}'$ is equivalent to forcing with $\mathbb{P}$ below~$p$. In particular, there is a $\mathbb{P}'$-name $\dot D'$ for the same ultrafilter adjoined by $\mathbb{P}$, and $(\mathbb{P}',\dot D')$ will be nice to $\kappa$. Given $q\in\mathbb{P}'$, the ideal $J_q$ is exactly the same whether computed using $\mathbb{P}$ or $\mathbb{P}'$, so there will be an $\langle i, j\rangle$-saturated pair of sets as required.
\end{proof}

We close this subsection with a definition and a corollary that summarize the work we have done.

\begin{definition}
A collection $\mathbb{J}$ of normal ideals on $\omega_1$ is $\langle i, j\rangle$-saturated if for any $J\in \mathbb{J}$ we can find a pair of $J$-positive sets that are $\langle i, j\rangle$ saturated over $J$.
\end{definition}

\begin{corollary}
\label{saturation}
If there is a generic elementary embedding with critical point $\omega_1$ that is well-founded out to the image of $\beth_2(\omega_1)^+$, then for any $l<\omega$ and coloring $c:[\omega_1]^2\rightarrow l$ with $l<\omega$ for some pair of colors $\langle i, j\rangle$ there is an everywhere weakly precipitous $\langle i, j\rangle$-saturated family of normal ideals on $\omega_1$.
\end{corollary}

\subsection{Bringing in topology}

We turn now to adapting another idea, that of $\langle i, j\rangle$--winners, from~\cite{dilipstevo} to our context.  This is where the topology on $\omega_1$ (viewed as a subset of $\mathbb{R}$) will be used.

\begin{observation}
\label{obssmall}
    Suppose $J$ is a countably complete ideal on $\omega_1$ and $n<\omega_1$.  If $A$ is $J$-positive then there is a $J$-positive $B\subseteq A$ that is covered by an open interval of $\mathbb{R}$ of length less than $1/n$.
\end{observation}
\begin{proof}
Given $n$, we can cover $A$ with countably many intervals of length $<1/n$ just using the topology of $\mathbb{R}$.  Since $J$ is countably complete, one of the corresponding pieces of $A$ must be $J$-positive.
\end{proof}

\begin{lemma}
\label{winner}
Suppose $\mathbb{J}$ is an everywhere weakly precipitous $\langle i, j\rangle$-saturated family of normal ideals on $\omega_1$.  For any $J\in\mathbb{J}$, for $J$-almost every $\alpha$ we can find objects $\langle I_n:n<\omega\rangle$ and $\langle T_n:n<\omega\rangle$ such that
\begin{enumerate}
    \item $I_n$ is an extension of $J$ in $\mathbb{J}$,
    \sk
    \item $T_n\in I_n^*$,
    \sk
    \item $\langle I_k, I_n\rangle$ is $\langle i, j\rangle$-saturated if $k>n$,
    \sk
    \item $c(\alpha,\beta)=i$ for each $\beta\in\bigcup_{n<\omega}T_n$, and
    \sk
    \item every open neighborhood of $\alpha$ contains $T_n$ for all sufficiently large $n$.
\end{enumerate}
\end{lemma}
\begin{proof}
It suffices to prove that every $J$-positive set $A$ contains such a point $\alpha$.  Given $A$, we define a strategy for \pone in the game $\Game(\mathbb{J})$ played beyond $J$.   The initial move will be to select $A_0=A$.   Given \ptwo\!\!'s move $J_n$, \pone will choose sets $A_{n+1}$ and $B_{n+1}$ such that
\begin{itemize}
    \item $\langle A_{n+1}, B_{n+1}\rangle$ is an $\langle i, j\rangle$-saturated pair over $J_n$,
    \sk
    \item $A_{n+1}\cup B_{n+1}\subseteq A_n$,
    \sk
    \item $A_{n+1}$ is covered by an open interval of length less than $1/n$, and
    \sk
    \item every $\alpha\in A_{n+1}$ is $i$-large in $B_{n+1}$ over $J_n$.
\end{itemize}
This is easily accomplished: given $J_n\in\mathbb{J}$ we can find $J_n$-positive sets $A'$ and $B'$ that are $\langle i, j\rangle$-saturated over $J_n$.  Since the rules of our game guarantee that $A_n$ is in the dual filter $J_n^*$, we may assume that both $A'$ and $B'$ are subsets of $A_n$. We let $B_{n+1}=B'$ and obtain $A_{n+1}$ by applying Observation~\ref{obssmall} to $A'$.

The strategy does not win for \pone, there is a run of the game for which $\bigcap_{n<\omega}A_n\neq\emptyset$, and we claim that any $\alpha$ in this intersection will have the required properties.\footnote{Note that the intersection will be a singleton since $A_n$ has diameter less than $1/n$, so in fact this $\alpha$ is unique.}.

To see why, define $T_n$ to be the $J_n$-positive set $\{\beta\in B_{n+1}: c(\alpha,\beta) = i\}$, and let $I_n$ be an extension of $J_n\restr T_n$ in $\mathbb{J}$. It is easy to check that this choice works, with perhaps everything but requirement (3) immediate from the construction.  To see (3), note that is $k>n$ then $I_k$ and $I_n$ are extensions of $J_n\restr A_{n+1}$ and $J_n\restr B_{n+1}$ in $\mathbb{J}$, and hence the pair of ideals $\langle I_k, I_n\rangle$ is $\langle i, j\rangle$-saturated by Lemma~\ref{idealsat}.
\end{proof}

We now summarize this section with a definition and theorem that provide what we need to push through the Raghavan-Todor\v{c}evi\'c construction.

 \begin{definition}
Let $\mathbb{J}$ be a collection of normal ideals on $\omega_1$.
\begin{enumerate}
    \item Given $J\in\mathbb{J}$, a point $\alpha<\omega_1$ is an $\langle i, j\rangle$-winner for $J$ in $\mathbb{J}$ if the conclusion of Lemma~\ref{winner} holds.
\sk
    \item We say that $\mathbb{J}$ is an $\langle i, j\rangle$-winning family of normal ideals if for each $J\in \mathbb{J}$
    \begin{itemize}
    \sk
        \item $J$-almost every $\alpha<\omega_1$ is an $\langle i, j\rangle$ winner for $J$ in $\mathbb{J}$, and
        \sk
        \item for each $A\in J^+$ there is an ideal in $\mathbb{J}$ extending $J\restr A$.
    \end{itemize}
    \sk
    \item Given a coloring $c:[\omega_1]^2\rightarrow l$, a pair of colors $\langle i, j\rangle$ is a winning pair for~$c$ if there is a non-empty $\langle i, j\rangle$-winning family of normal ideals on $\omega_1$.
\end{enumerate}
 \end{definition}

\begin{theorem}
If there is a  generic elementary embedding with critical point $\omega_1$ that is well-founded out to the image of $\beth_2(\omega_1)^+$, then for any set $X\subseteq\mathbb{R}$ of cardinality~$\aleph_1$ and coloring $c:[X]^2\rightarrow l$ with $l<\omega$, there is a winning pair of colors for $c$.
\end{theorem}
\begin{proof}
Immediate based on our work.  From the existence of the embedding we can find an everywhere weakly precipitous family of ideals that is $\langle i, j\rangle$-saturated for some pair of colors $\langle i, j\rangle$, and Lemma~\ref{winner} tells us that this suffices.
\end{proof}

\section{The Raghavan-Todor\v{c}evi\'c Construction}

Now our challenge is to show that we can use a winning pair of colors $\langle i, j\rangle$ to build a copy of the rationals on which our coloring assumes only the values $i$ or $j$. This can be done through a modification of the argument from~\cite{dilipstevo} tailored to take advantage of our hypothesis.

\begin{theorem}[The Raghavan-Todor\v{c}evi\'c Construction]
\label{RT}
Assume $X\subseteq\mathbb{R}$ is of cardinality $\aleph_1$ and $c:[X]^2\rightarrow l$ with $l<\omega$.  If there is a winning pair of colors $\langle i, j\rangle$ for $c$, then  $X$ has a subset $Y$ homeomorphic to $\mathbb{Q}$ on which $c$ assumes only the colors $i$ and $j$.
\end{theorem}

\begin{proof}
We identify $X$ and $\omega_1$ as usual and let $\mathbb{J}$ be an $\langle i, j\rangle$-winning family of normal ideals.  We will build $Y$ using a construction with $\omega$ stages, associating a point of $X$ to each element of $\pre{\omega}{<\omega}$.
 associating a point $\alpha_\sigma$ of $X$ to each $\sigma\in\pre{\omega}{<\omega}$.  We will use the lexicographic ordering of    $\pre{\omega}{<\omega}$ to organize things, so recall that if $\sigma$ and $\tau$ are in $\pre{\omega}{<\omega}$ then $\sigma<_\lex\tau$ means that $\tau$ is a proper extension of $\sigma$ or $\sigma(i)<\tau(i)$ for the least $i$ where the sequences differ.

To see how this will work, fix a one-to-one enumeration $\langle\sigma_n:n<\omega\rangle$ of $\pre{\omega}{<\omega}$ chosen so that a sequence will not be enumerated until after all of its proper initial segments have already appeared, and we will let $\alpha_n$ denote the point chosen for the sake of $\sigma_n$.   We will ensure that these points are all distinct, and aim for the following two goals:

\medskip
\noindent{\sf Goal 1:}  For $k<n$, $c(\alpha_k, \alpha_n)$ will equal $i$ if $\sigma_k<_\lex \sigma_n$, and  and $c(\alpha_k, \alpha_n)$ will equal $j$ if $\sigma_n<_\lex \sigma_k$.\footnote{In other words, once we have picked $\alpha_k$ then for $n>k$ we promise to choose $\alpha_n$ so that if it is (lexicographically) to the right of $\alpha_k$ then $c(\alpha_k, \alpha_n)=i$, and if it is to the left then $c(\alpha_k,\alpha_n)=j$.}

\medskip

\noindent{\sf Goal 2:} If $\{k_m:m<\omega\}$ are the indices corresponding to the immediate successors of $\sigma_n$ in $\pre{\omega}{<\omega}$, then $\{\alpha_{k_m}:m<\omega\}$ converges to $\alpha_n$.
\medskip

\noindent  Note that this will be sufficient, as $Y=\{\alpha_n:n<\omega\}$ will be a countable dense-in-itself subset of $X$ on which $c$ takes on at most two values.

For $n<\omega$, we let $Q_n$ to consist of the sequences $\sigma_k$ for $k<n$ as well as all of their immediate successors in $\pre{\omega}{<\omega}$.  With this arrangement, we can say the following:
\begin{itemize}
    \item $Q_0$ is the empty set and $\sigma_0$ is the empty sequence. For $n>0$, $Q_n$ is a downward closed subtree of $\pre{\omega}{<\omega}$ of finite height.
    \sk
    \item If $\sigma=\sigma_k$ for $k<n$, then $Q_n$ will contain $\sigma^\smallfrown\langle m\rangle$ for all $m<\omega$.  If $\sigma\in Q_n$ is not of the form $\sigma_k$ for $k<n$, then no proper extension of $\sigma$ is in $Q_n$, and $\sigma$ is of the form $\sigma_k^\smallfrown\langle m\rangle$ for some $k<n$.
    \sk
    \item The sequence $\sigma_n$ is an element of $Q_n$, but no proper extension of $\sigma_n$ is in $Q_n$.
\end{itemize}

Thus, the nodes of $Q_n$ are divided into {\em leaves $L_n$} and {\em branching nodes $B_n$} with the usual meaning, and $B_n$ will be exactly the set $\{\sigma_k:k<n\}$. As noted above, the sequence $\sigma_n$ will be in $Q_n$, and since our enumeration is one-to-one we know that $\sigma_n$ will be a leaf of $Q_n$.
At stage $n$ of the construction, we will be choosing the point $\alpha_n$ assigned to $\sigma_n$,  and we assume that our construction has provided us with the following circumstances:

\begin{enumerate}[label=(\alph*)]
    \item The set $\{\sigma_k:k<n\}$ of sequences enumerated so far lists the branching nodes $B_n$ of $Q_n$, and for each $k<n$ a point $\alpha_k$ has been assigned to $\sigma_k$. In addition, each leaf $\tau$ in $L_n$ has been assigned an ideal $J(\tau, n)$ from $\mathbb{J}$, as well as a set $T(\tau, n)$ from the dual filter $J^*(\tau,n)$.
    \sk
    \item If $\tau$ is a leaf of $Q_n$ then for any $\beta\in T(\tau, n)$ and $k<n$,
    \begin{equation*}
    c(\alpha_k, \beta)=
    \begin{cases}
    i &\text{if $\sigma_k<_\lex\tau$,}\\
    j &\text{if $\tau<_\lex\sigma_k$}.
    \end{cases}
\end{equation*}
    \item If $\sigma$ and $\tau$ are distinct leaves in $Q_n$ with $\sigma<_\lex \tau$, then the pair of ideals $\langle J(\sigma, n), J(\tau, n)\rangle$ is a $\langle i, j\rangle$-saturated pair of normal ideals.
    \sk
    \item For $k<n$,  the sequence of sets $\langle T(\sigma_k^\smallfrown\langle m\rangle, n):\sigma_k^\smallfrown\langle m\rangle\in L_n\rangle$ attached to the successors of $\sigma_k$ in $L_n$ converges to~$\alpha_k$.
\end{enumerate}

Getting the construction started is easy, as our requirements give us freedom to choose $J(\sigma_0, 0)$ to be any ideal in $\mathbb{J}$ and let $T(\sigma_0, 0)=\omega_1$.  For larger $n$, our enumeration will hand us the sequence $\sigma_n$ (a leaf of $Q_n$) and we will choose the corresponding point $\alpha_n$ from the current crop of candidates $T(\sigma_n, n)$.  Because of assumption (b), this will maintain Goal~1 for $k<n$. There are some other restrictions on $\alpha_n$, though: we need to make sure that we will be able to find  $J(\tau,n+1)$ and $T(\tau, n+1)$ for each leaf $\tau$ of $Q_{n+1}$. Note that such $\tau$ are either leaves of $Q_n$ (for which $J(\tau, n)$ and $T(\tau, n)$ have already been defined), or else newly created leaves of the form $\sigma_n^\smallfrown\langle m\rangle$ for some $m<\omega$.

To make sure we can do this, we want to choose $\alpha_n\in T(\sigma_n, n)$ that satisfies the following:
\begin{enumerate}[label=(\alph*),resume]
\item If $\tau\neq\sigma_n$ is another leaf of $Q_n$, then
\sk
\begin{itemize}
\item $\sigma_n<_\lex\tau\Longrightarrow\{\beta\in T(\tau,n):c(\alpha_n, \beta)=i\}$ is $J(\tau, n)$-positive.
\sk
    \item  $\tau<_\lex \sigma_n\Longrightarrow\{\beta\in T(\tau,n):c(\alpha_n,\beta)=j\}$ is $J(\tau, n)$-positive, and
\end{itemize}
\sk
\item $\alpha_n$ is an $\langle i, j\rangle$ winner for $J(\sigma_n, n)$ in $\mathbb{J}$.
\end{enumerate}
It will be (e) that allows us to find $J(\tau, n+1)$ and $T(\tau, n+1)$ for those leaves of $Q_{n+1}$ that come from $Q_n$, and it will be (f) that allows us to define the objects for leaves that arise as successors of $\sigma_n$.  Thus, the following claim is key.

\begin{claim}
$J(\sigma_n, n)$-almost every $\alpha\in T(\sigma_n, n)$ satisfies both (e) and (f).
\end{claim}
\begin{proof}
Condition (f) is a consequence of Lemma~\ref{winner}.  To see that (e) holds, we need to use our assumption (c) and the $\langle i, j\rangle$-saturation of the relevant ideals.  This tells us that given another leaf $\tau$ of $Q_n$, then
\begin{align}
 \sigma_n<_\lex\tau&\Longrightarrow   (\forall^{J(\sigma_n, n)}\alpha\in T(\sigma_n, n))(\exists^{J(\tau,n)}\beta\in T(\tau, n))\left[c(\alpha,\beta)=i\right]\\
\intertext{and}
 \tau<_\lex\sigma_n&\Longrightarrow   (\forall^{J(\sigma_n, n)}\alpha\in T(\sigma_n, n))(\exists^{J(\tau,n)}\beta\in T(\tau, n))\left[c(\alpha,\beta)=j\right]
\end{align}
Informally it says that for any other leaf $\tau$ of $Q_n$, almost every element of $T(\sigma_n, n)$ is ``suitably connected'' (in the sense that the correct sets are large) to $T(\tau, n)$.  Since $J(\sigma, n)$ is countably complete, it follows that almost every element of $T(\sigma, n)$ is suitably connected to $T(\tau, n)$ for {\em ALL} other leaves $\tau$ of $Q_n$, and so claim follows.
\end{proof}

Thus, we should choose $\alpha_n\in T(\sigma_n, n)$ satisfying (e) and (f) and now we show how to define $T(\tau, n+1)$ and $J(\tau, n+1)$ for the leaves of $Q_{n+1}$ so that our conditions (a)-(d) will be maintained. Given such a $\tau$ that comes from $Q_n$, we use (f) to define $T(\tau,n+1)$:
\begin{align}
\label{T1}
 \sigma_n<_\lex\tau&\Longrightarrow T(\tau, n+1)=\{\beta\in T(\tau, n):c(\alpha_n, \beta)=i\}\\
\intertext{and}
\label{T2}
 \tau<_\lex\sigma_n&\Longrightarrow T(\tau, n+1)=\{\beta\in T(\tau, n): c(\alpha_n,\beta)=j\}.
\end{align}
In either case, we know that $T(\tau, n+1)$ will be $J(\tau, n)$-positive because $\alpha_n$ was suitably connected to $T(\tau, n)$, and given this we are able to choose $J(\tau, n+1)$ to be any ideal in $\mathbb{J}$ that extends~$J(\tau, n)\restr T(\tau, n+1)$.

It remains to assign sets and ideals to the sequences $\sigma_n^\smallfrown\langle m\rangle$ that will appear as new leaves in $Q_{n+1}$. This is where we use our assumption that $\alpha_n$ is an $\langle i, j\rangle$-winner: if $\langle I_m:m<\omega\rangle$ and $\langle T_m:m<\omega\rangle$ are as in the conclusion of Lemma~\ref{weapon}, then we define
\begin{align}
    J(\sigma^\smallfrown\langle m\rangle, n+1)&=I_m\\
    \intertext{ and }
    T(\sigma^\smallfrown\langle m\rangle, n+1)&=T_m.
\end{align}

To finish, we need to check that (a) through (d) hold with $n$ replaced by $n+1$. Condition (a) is immediate.  For (b), we again use the fact that the leaves of $Q_{n+1}$ either come from $Q_n$, or were added as an immediate successor of $\sigma_n$. Given a leaf $\tau$ of the first sort, we know $T(\tau, n+1)$ is a subset of $T(\tau, n)$, and if $\tau$ is of the second sort then we know $T(\tau, n+1)$ is a subset of $T(\sigma_n, n)$. This is enough to establish (b) for the case where $k<n$.  If $\tau$ is an immediate successor of $\sigma_n$, then we know (b) holds for $k=n$ as this was part of the conclusion of Lemma~\ref{winner}.  If on the other hand $\tau$ is a leaf from $Q_n$, then (b) holds because of how we chose $T(\tau,n+1)$ back in (\ref{T1}) and (\ref{T2}).   Condition (c) holds because our construction guarantees $\langle J(\sigma, n+1), J(\tau,n+1)\rangle$ extends an $\langle i, j\rangle$-saturated pair of ideals.  Finally, (e) holds for $\sigma_n$ because this is part of being an $\langle i, j\rangle$-winner, and it continues to hold for $k<n$ because the required convergence is unchanged when we shrink the sets as we pass from $n$ to $n+1$. Thus, the construction can be maintained through stage $n+1$, and as noted above, $Y=\{\alpha_n:n<\omega\}$ will be as required.
\end{proof}

As a corollary, we have our theorem promised in the introduction:

\begin{mainthm}
If there is a Ramsey cardinal (or more generally, a generic elementary embedding with critical point $\omega_1$ that is well-founded out to the image of $\beth_2(\omega_1)^+)$)  then for any uncountable set of reals $X$ and any $c:[X]^2\rightarrow l$ with $l<\omega$, there is a set $Y\subseteq X$ homeomorphic to $\mathbb{Q}$ for which the range of $c\restr [Y]^2$ contains at most two colors.
\end{mainthm}

\end{document}